\documentclass{ws-jaa}

\begin{document}

\newcommand{\F}{\mathbb{F}_q}
\newcommand{\Fp}{\mathbb{F}_p }
\newcommand{\Ffive}{\mathbb{F}_5}
\newcommand{\Ftwofive}{\mathbb{F}_{25}}
\newcommand{\ds}{\displaystyle}
\newcommand{\la}{\lambda}
\newcommand{\m}{\mathrm{mod}\;}
\newcommand{\T}{\mathcal{T}}
\newcommand{\Q}{\mathcal{Q}}

\newcommand{\al}{\alpha}
\newcommand{\be}{\beta}
\newcommand{\gl}{\gamma}
\numberwithin{equation}{section}

\markboth{Stephen D. Cohen}{Triples and quadruples of consecutive squares or non-squares}

\catchline{}{}{}{}{}
%
%%%%%%%%%%%%%%%%%%%%%%%%%%%%%%%%%%%%%%%%%%%%%%%%%%%%%%%%%%%%%%%%%%%%

\title{TRIPLES AND QUADRUPLES OF CONSECUTIVE SQUARES\\  OR NON-SQUARES IN A FINITE FIELD }

\author{\footnotesize STEPHEN D. COHEN\footnote{Mailing addres: 6 Bracken Road, Portlethen, Aberdeen AB12 4TA, Scotland, UK}}

\address{{PROFESSOR EMERITUS, UNIVERSITY OF GLASGOW, UK} \\
\email{stephen.cohen@glasgow.ac.uk}}

\maketitle

\begin{history}
\received{(Day Month Year)}
\revised{(Day Month Year)}
\accepted{(Day Month Year)}
\comby{(xxxxxxxxx)}
\end{history}

\begin{abstract}Let $\F$ be the finite field of odd prime power order $q$.  We
 find explicit expressions for the number  of triples $\{\al-1,\al,\al+1 \}$ of consecutive non-zero squares   in $\F$ and similarly for the number of  triples of consecutive non-square elements.  A key ingredient is the evaluation of Jacobsthal sums over general finite fields  by Katre and Rajwade. This extends  results of Monzingo (1985)  to non-prime fields.
Curiously, the same machinery allows the evaluation of the number of consecutive quadruples    $\{\al  -1, \al,\al+1, \al +2\}$ of square and non-squares over $\F$, when $q$ is a power of 5.
\end{abstract}


\ccode{2010 Mathematics Subject Classification: 11T24, 11T30, 16U99}

\section{Introduction}
Let $\F$ the  finite field of order $q=p^d$, where $q$ is an odd prime.  The elements of $\F$ comprise  $\frac{q-1}{2}$ elements that are non-zero squares  of elements of $\F$ (quadratic residues)  $\frac{q-1}{2}$ non-square elements (quadratic non-resdues)  together with $0$. 

 In brief, the purpose of this note is to provide an expression for the number of triples of consecutive elements of $\F$  which are either all non-zero squares or all non-squares.   To clarify this we summarise  what is meant by consecutive elements.  If  $q=p$, an odd prime,   the members of $\Fp$ possess a natural cyclic  ordering $(0,1,2, \ldots, p-1)$ and a run of {\em  distinct} consecutive elements $\{\al, \al+1, \ldots, \al+\ell -1\}$ has length $\ell$ not exceeding $p$.    For a general prime power $q=p^d$, if $\al, \be  \in \F$ are defined to be equivalent when $\al- \be \in \Fp$,  then $\F$ is divided into $p^{d-1}$ equivalence class of the form $\al+\Fp$,  whose elements can be cyclically ordered corresponding to the ordering  of the elements of  $\Fp$.   Again, the length of a maximum  run of consecutive elements of $\F$ is $p$. 

In particular, given any element $\al \in \F$, call  $\T_\al $ the triple of consecutive elements $\{\al-1,\al,\al+1\}$ in $\F$.     Define $M_q$ to be the number of $\al \in \F$ such that the members of $\T_\al$ are all  non-zero squares.  Similarly, define $N_q$ to be the number of $\al \in \F$ such that the members of $\T_\al$ are all non-squares.  In the situation in which  $q=p$, a prime,  Monzingo \cite{M86}  evaluated  each of $M_q$ and $N_q$ in four  cases.   Our theorem subdivides one of these into two  (Cases 4 and  5 below) to cater for square prime powers  $q$.

\begin{theorem} \label{main}
Let $q=p^d$, where $p$ is an odd prime.

\medskip
\noindent
{\bf \em Case 1.} If  $q \equiv 7\   (\m 8)$, then
\begin{equation}\label{Case1}
M_q=N_q = \frac{q-7}{8}.
\end{equation}

\noindent
{\bf  \em Case 2.} If $q \equiv 3\   (\m 8)$, then
\begin{equation}\label{Case2}
M_q=N_q = \frac{q-3}{8}.
\end{equation}

\medskip
\noindent
{\bf \em Case 3.}
If $q \equiv 5\ (\m 8)$, then
\begin{equation}\label{Case3}
M_q= \frac{q+2s-7}{8}; \quad N_q = \frac{q-2s-3}{8},
\end{equation}
where  integers $s,t$   are uniquely defined by  $q=s^2+t^2, p\nmid s,  s \equiv 1\ (\m 4)$.

\medskip
\noindent
{\bf  \em Case 4.}
If  $q =p^d \equiv 1\ (\m 8)$, where $p \equiv 1 \ (\m 4)$, then 
\begin{equation}\label{Case4}
M_q= \frac{q-2s-15}{8}; \quad N_q = \frac{q+2s-3}{8},
\end{equation}
where integers $s,t$ are defined as in Case $3$.

\medskip\noindent
{\bf \em Case 5.}
If  $q =p^d \equiv 1\ (\m 8)$, where $p \equiv 3 \ (\m 4)$,  then $q$ is a square and
\begin{equation}\label{Case5}
  M_q =\frac{1}{8}\left(q -2 (-1)^{d/2}\sqrt{q}-15\right); \quad N_q =\frac{1}{8}\left(q +2 (-1)^{d/2}\sqrt{q}-3\right). 
\end{equation}
\end{theorem}

\medskip
A key ingredient to our proof of Theorem \ref{main} is the extension by Katre and Rajwade in 1987 \cite{KR87} to the classical evaluation of the Jacobsthal sum, see Lemma \ref{KR}. 

Before proceeding, we conmment on related work  on triples of inteers modulo a prime in the literature.       Monzingo \cite {M85} derives his version of Theorem \ref{main} for triples  of quadratic residues and non-residues modulo a prime $p$ by elementary means.  In \cite{M86}  he shows the connection with Jacobsthal sums by deriving the latter out of considerations of the distribution of quadratic residues and non-residues modulo a prime  $p$.  Further, Theorem 2 of Sun  \cite{S02}   gives a description of the set of all triples   $\T$  comprising only non-zero squares or only non-squares  of integers modulo $p$.    The cardinality of this set is in agreement with our  values of the sum $M_q+N_q$ (in Cases 1-4).

\medskip
The author's work on Theorem \ref{main} in respect of consecutive triples of non-squares was originally  presented in \cite{ACDT23}. Its  motivation  came from commutative ring theory. Let $R$ be an integral ring (or domain), see \cite{ACDT23} and Theorem \ref{ACDT}, below.    An element $\al  \in R$ is called a $k$-potent if $\al^k=\al$.  Use of Theorem \ref{main}  (not just for prime fields)  allows the  classificaton of integral rings for which for some positive integer $n$,  each matrix in $M_n(R)$  is the sum of a tripotent  (3-potent) and a $k$-potent. Counting consecutive triples (or longer runs)  of squares and non and non-squares is also relevant to counting points on elliptic curves over $\F$, see \cite{KTVZ}.

\bigskip
Following  Theorem \ref{main}, one  might contemplate the evaluation of consecutive quadruples (say,$ \{\al-1, \al, \al+1, \al+2 \}$)  in $\F$ that are either all non-zero squares or all non-squares, although one would suspect this would involve the evaluation of character sums of higher degree. (Here to insist that the quadruples contain {\em distinct} elements one would assume that the characteristic $p$ is at least 5).     Surprisingly, perhaps, when $p=5$ (and only then), we are able to evaluate these numbers using nothing  beyond  Lemma \ref{KR}. 

\begin {theorem}\label{5thm}
Let $q=5^d$ and $m_q, n_q$ denote respectively, the number of  quadruples $ \{\al-1, \al, \al+1, \al+2 \}$ of consecutive members of $\F$ that are all non-zero squares or all non-squares.   Define integers $s, t$  as in Case $3$ of Theorem $\ref{main}$.   Then, if $d$ is even, we have
 \begin{equation}\label{5even}
m_q=\frac{q-10s-39}{16};\quad n_q=\frac{q+6s-7}{16}.
\end{equation}

If, on the other hand, $d$ is odd, then
\begin{equation} \label{5odd}
m_q=n_q= \frac{q+2s-7}{16}.
\end {equation}

\end{theorem} 

\section{Character sums}

 Suppose throughout $q=p^d$ is an odd prime power and let $\la$ be the quadratic character on $\F$. Thus

\begin{equation}\label{sumla}
\sum_ \alpha\la(\alpha)=\sum_{\alpha \neq 0}\la(\alpha)=0,
\end{equation}
where $\ds{\sum_{\alpha}}$ stands for $\ds{\sum_{\alpha \in \F}}$ and $\ds{\sum_{\alpha \neq 0}}$ means that $\alpha =0$ is excluded from the sum. We need further evaluations of character sums. The first can be found in \cite[Theorem 2.1.2]{BEW98}.

\begin{lemma}\label{BEW}
Suppose $q$ is an odd prime power. Let $f(x)=x^2+bx+c \in \F[x]$. Assume $b^2-4c \neq 0$. Then
\[\sum_{\alpha \in \F}\la(f(\alpha))=-1.\]
\end{lemma}
The next result concerns the Jacobsthal sum $\ds{J(a) =\sum_{\alpha }\la(x(x^2+a)), a \in \F}$. Recall that
  \[ \la(-1)= \begin{cases} 1,  &\text{if } q \equiv 1 \ (\m 4),\\
  -1,  &\text{if } q \equiv 3\  (\m 4).
     \end{cases}     \]
Hence, by replacing $\alpha$ by $-\alpha$ in the expression for $J(a)$, we see that, if $ q \equiv 3\ (\m 4)$, then $J(a)=0$. On the other hand, when $q \equiv 1 \ (\m 4)$, we use an  extension of the classical evaluation of the Jacobsthal sum on prime fields given by Katre and Rajwade \cite{KR87} For any  $q \equiv 1 \ ( \m 4)$, they evaluate $J(a)$ for any non-zero $a \in \F$.  We only require the case in which $a$ is a square in $\F$.

\begin{lemma}[\cite{KR87}, Theorem 2]\label{KR}
Suppose that  $q=p^d \equiv 1 \ (\m 4)$, where $p$ is an odd prime and $a$ is a square in $\F$.    If $p \equiv 3\ (\m 4)$ (so that $d$ is even), let $s=(-1)^{d/2}\sqrt{q}$. If $p \equiv 1\ (\m 4)$, define $s$ uniquely by  $q=s^2+t^2, p \nmid s, s \equiv 1 \ (\m 4)$. Then
\[ J(a) = \begin{cases} -2s,  & \mbox{ if $a$ is a fourth power in $\F$,}\\
2s, & \mbox{ if $a$ is a square but not a fourth power in $\F$.}
\end{cases} \]
\end{lemma}

\section {Proof of Theorem \ref {main}}\label{secproof}

Obviously,  a triple $\T_\al$  cannot comprise only  non-zero squares or  only non-squares if $\al =0, \pm 1$.  Hence, we avoid consideration of such triples.

As a further preliminary, observe that, if  $q \equiv 3\ (\m 4)$ and  if $\T_\al$ is a triple consisting of squares, then $-T_{-\al}$ is a triple   of non-squares.   Conseqently, in this case $N_q=M_q$

   More generally, for any odd $q$,  we have

\begin{equation}\label{Meq}
M_q  = \frac{1}{8}\sum_{\alpha \neq 0, \pm 1}(1+\la(\alpha))(1+\la(\alpha-1))(1+\la(\alpha+1) ), 
\end{equation}
and 
\begin{equation}\label{Neq}
N_q  = \frac{1}{8}\sum_{\alpha \neq 0, \pm 1}(1-\la(\alpha))(1-\la(\alpha-1))(1-\la(\alpha+1) ).
\end{equation}

Now, set
\[ S_1=\sum_{\alpha \neq 0, \pm1}\la(\alpha); \quad S_2=\sum_{\alpha \neq 0, \pm1}\la(\alpha -1); \quad S_3=\sum_{\alpha \neq 0, \pm1}\la(\alpha+1)\]
  and
\[T_1= \sum_{\alpha \neq 0, \pm 1}\la(\alpha(\alpha-1)); \;T_2= \sum_{\alpha \neq 0, \pm 1}\la(\alpha(\alpha+1)); \;  T_3= \sum_{\alpha \neq 0, \pm 1}\la(\alpha^2-1).   \]

  Then
  \[M_q= \frac{1}{8}\left(q-3 + \sum_{i=1}^{3} S_i+ \sum_{i=1}^3T_i + J(-1)\right) \]
  and
\[N_q= \frac{1}{8}\left(q-3 - \sum_{i=1}^{3} S_i+ \sum_{i=1}^3T_i - J(-1)\right). \]

  From (\ref{sumla}),
\[S_1 =\sum_\alpha \la(\alpha)- 1 -\la(-1)=-1 - \la(-1); \; S_2=-\la(-2) - \la(-1); \; S_3= -1-\la(2), \]
whereas, from Lemma \ref{BEW},
\[ T_1=\sum_{\alpha \neq -1}\la(\alpha(\alpha-1))=-1-\la(2); \; T_2= -1- \la(2); \; T_3= -1- \la(-1). \]
Further,  $J(-1)=0$ if $q  \equiv 3\ (\m 4)$. But, when  $q \equiv 1\ (\m  4)$, by Lemma \ref{KR},  we have

\begin{equation}\label{Jeq}
J(-1) = \begin{cases} -2s, & \mbox{if } q \equiv 1\ (\m 8),\\
                                       2s, & \mbox{if } q \equiv 5\ (\m 8).
                                       \end{cases} 
\end{equation}
We also have the well-known facts that
\begin{equation}\label{l2}
\la(2)=\begin{cases} 1, & \mbox{if } q \equiv \pm 1 \ (\m 8),\\

                                                                                                       -1, &  \mbox{ if } q \equiv \pm 3 \ (\m 8),
 \end{cases}
  \end{equation}
and
  \begin{equation}\label{l-2}\la(-2)=\begin{cases} 1, & \mbox{if } q \equiv  1 \mbox{ or } 3\ (\m 8),\\

                                                                                                       -1, &  \mbox{ if } q \equiv  5 \mbox{ or }7 \ (\m 8).
                                                                                                       \end{cases} 
                                                                                                      \end{equation}

To evaluate $M_q, N_q$ from (\ref{Neq}), (\ref{Jeq}), (\ref{l2}) and (\ref{l-2}), we consider in turn the  five cases of Theorem \ref{main} and illustrate the result for some small values of $q$.

\noindent
{\bf Case 1:} {\it If $ q \equiv 7\  (\m 8)$,  then $\ds{N_q =\frac{q-7}{8}}.$
\begin{proof}
  Here  $\la(-1)=-1, \la(2) =1, \la(-2)=-1$. Thus $S_1 =0, S_2= 2, S_3=-2, T_1=T_2=-2, T_3=0 $, while $J(-1)=0$.  Hence
  \[8 N_q=( q-3 +0-4)= q-7.\]
\end{proof}

{\em Small examples of Case 1 include $M_7=N_7=0,M_{23}= N_{23} =2$.}

\medskip
\noindent
{\em \bf Case 2:} {\it If $ q \equiv 3\  (\m 8)$,  then $\ds{N_q =\frac{q-3}{8}}.$

\begin{proof} Now $\la(-1)=-1, \la(2)=-1, \la(-2)=1$. Thus $S_1= S_2=S_3=T_1=T_2=T_3=0$. Also, $J(-1)=0$. The reult is now immediate.
\end{proof}

{\em Small examples of Case 2 include $M_3=N_3=0,M_{11}= N_{11}=1,M_{19}= N_{19} =2$.}

\medskip
\noindent
{\em \bf Case 3:} {\it If $ q \equiv 5\  (\m 8)$, then}
  \[ M_q=\frac{q+2s-7}{8}q; \quad N_q =\frac{q-2s-3}{8}, \]
{\it where $q=s^2+t^2,\   p \nmid s, \  s \equiv 1 \ (\m 4)$ (as in Lemma $ \ref{KR})$.}

\begin{proof}
Here $q=p^d$, where also $p \equiv 5\ (\m 8)$ and $r$ is odd. We have $\la(-1)=1, \la(2)= \la(-2)=-1$. Hence, $S_1=-2, S_2= S_3=0, T_1=T_2=0, T_3=-2$.

Further, let $\al$ be a primitive element in $\F$. Then $-1= \al^\frac{q-1}{2}$ is the square of $\al^\frac{q-1}{4}$ but not a fourth power, since $\frac{q-1}{4}$ is odd.
Hence  $J(-1)=2s$.
\end{proof}

{\em Small examples of Case 3 include

  $M_5=N_5=0$ (since $5=1^2+2^2$),

 $M_{13}=0, \  N_{13} =2$ (since $13=(-3)^2+2^2$),
 
  $M_{29}=4,\ N_{29}=2$ (since $29=5^2+2^2$),
  
  $M_{125}=12, \  N_{125}=18$  (since $125=(-11)^2 +2^2 $, but  ignoring $125=5^2+10^2$ because $p\nmid s$).}

\medskip
\noindent
{\em \bf Case 4:} {\it If $ q=p^d \equiv 1\  (\m 8)$, where $p \equiv 1\ (\m 4)$, then  }
     \[   M_q =\frac{q-2s-15}{8};\quad  N_q =\frac{q+2s-3}{8}, \]
{\it where $q=s^2+t^2,\ p\nmid s,\  s \equiv 1 \ (\m 4)$.}

\begin{proof}
Here  $\la(-1)= \la(2)=\la(-2)=1$. Hence, $S_1=S_2=S_3=T_1=T_2=T_3=-2$. This time $\frac{q-1}{4}$ is even and so $-1$ is a fourth power and $J(-1)=-2s$.
\end{proof}

{ \em  Small examples of Case 4 include 

$M_{17}=0,\ N_{17} =2$ (since $17=1^2+4^2$),

 $M_{25}=N_{25}=2$  (since $25=(-3)^2 +4^2$), 
 
  $M_{169}=18, \ N_{169}= 22$  (since $169=5^2+12^2$), 
  
  $M_{289}=38,\ N_{289}=32$ (since $289= (-15)^2+8^2$).}

\medskip
\noindent
{\em\bf Case  5:} {\it If $ q=p^d \equiv 1\  (\m 8)$, where $p \equiv 3 \ (\m 4)$, then $q$ is a square and}

\[M_q =\frac{1}{8}\left(q -2 (-1)^{d/2}\sqrt{q}-15\right),\quad N_q =\frac{1}{8}\left(q +2 (-1)^{d/2}\sqrt{q}-3\right). \]
\begin{proof} As in Case $4$, each $S_i$ and $T_i$ has the value $-2$. Again $(-1)$ is a fourth power in $\F$ so that, by Lemma $\ref{KR}$, $J(-1)=-2s=-2(-1)^{d/2}\sqrt{q}$.
\end{proof}

 {\em Small examples of Case $5$ include $M_9=N_9=0,\ M_{49}=6,\  N_{49}=4, \  M_{81}=6,\ N_{81} =12$}.

\em
\section{  Further results and remarks on triples}

Although the expressions for  $M_q$ and $N_q$ in Theorem \ref{main} determine these numbers precisely,  the expressions in Cases 3 and 4  do not depend explicitly on $q$ alone. To remedy this, we deduce lower bounds that depend more explicitly on $q$.  

\begin{corollary} \label{SDC1} Let $q=p^d$, where $p$ is an odd prime. Then
\begin{equation} \label{Nbound}
 N_q \ge \frac{1}{8}\Big(q -2\sqrt{q}-3\Big), 
\end{equation}
with equality if and only if  $p \equiv 3\;(\m 4)$ and $d=2m$, where $m$ is odd.
Further,
\begin{equation} \label{Mbound}
 M_q \ge \frac{1}{8}\Big(q -2\sqrt{q}-15\Big), 
\end{equation}
with equality if and only if $p \equiv 3\;(\m 4)$ and $d=4m$, where $m$ may be even or odd.
\end{corollary}
\begin{proof}  The bound $(\ref{Nbound}$ is vacuous unless $q \geq 9$ in which case, since $|s|< \sqrt{q}$, from Cases 1--5,  $ N_q$ exceeds the right side except when equality occurs in Case 5 with $n=2m$, $m$ odd. 

Similarly, $(\ref{Mbound})$ is vacuous unless $q \geq 25$  and then the right side exceeds the right side except in Case 5 with $4|d$.
\end{proof}

\begin{corollary}\label{MNpos}
For any odd prime power exceeding $9$ there exists a triple  of consecutive non-squares in $\F$.  Further, for any $q \notin\{3,5,7,9,13,17\}$, there exists a triple of consecutive non-zero squares.
\end{corollary}

\begin{proof}  Follows from Corollary \ref{SDC1} and the illusrasted values in the proof of Theorem \ref{main}.
\end{proof}

In connection with the first existence statement  of Corollary \ref{MNpos}, we remark that  it  is an immediate consequence of   Theorem 1 of \cite{CST15} for all odd prime powers $q$ exceeding 169.   Specifically, the latter asserts that  there exists  a triple  of consecutive primitive elements of $\F$ except when $q \in \{3, \ 5,\ 7,\ 9,\  13,\ 25,\ 29,\ 61,\  81,\  121, \ 169\}$.   The proof of that theorem is, however, far more intricate.

\medskip
Finally, the following is the application of  Theorem \ref{main} (in respect of $N_q$)  referred to in the introduction, see \cite{ACDT23}, Lemma 2.2.   
\begin{theorem}\label{ACDT}
Let $m>1$ be a positive integer and $R$ be  an integral ring of cardinality $q$.  If, for some positive integer $n$, each matrix in the matrix ring $M_n(R)$ is representable as a sum of a tripotent and a $m$-potent, then $R$ is a finite field $\F$ with $q$ an odd prime power and either  $q\le9$ and  ${\frac{q-1}{2}\big{|}( m-1)}$, or $q>9$ and $(q-1)| (m-1)$.
 \end{theorem}
Further ring-theoretical consequences are given in \cite{ACDT23}, Section 6.  Some background material is in  \cite{smz-2021} and \cite{laa-2021}.

\section{Proof of Theorem \ref{5thm} }
{\em }

Let $q=5^d$ so that $q \equiv 1$ or $ 5\ (\m 8)$ according as $d$ is even or odd, respectively.  The result and its proof are highly characteristic-dependent, so that the elements $\{0,1,2,3,4\}$  are respectively $\{0,1,2,-2,-1\}$. Moreover $\la (\pm1) =1$, while $\la(2)=\la(-2) =1 $ or $-1$ according as $d$ is even or odd, respectively.   

 Instead of triples we now seek to count quadruples $ \Q_\al= \{\al-1, \al, \al+1, \al+2\}$ of consecutive elements of $\F$ wholly consisting of non-zero squares (counted by $m_q$) or of non-squares (counted by $n_q$).  Since no element of a relevant quadruple $\Q_\al$ can be zero we can suppose $\al \neq 0, \pm 1, -2$.  We shall call the quadruple $ \Q_2= \{1,2,3,4\} $ of elements of $\Ffive$ the {\em basic quadruple}.  In particular, $\Q_2$ consists of non-zero squares if $d$ is even whereas $\Q_2$ contains two squares and two non-squares if $d$ is odd.

\begin{lemma} \label{M-N}
Suppose $q=5^d$, where $d$ is odd. Then $m_q=n_q$.
\end{lemma}
 \begin{proof}
First, observe that, since $p=5$, then  any 4-element subset of an equivalence class $\be  +\Ffive$ can be arranged as a quadruple $\Q_\al$  of consecutive elements.  As noted, the basic quadruple does not consist only of squares or only of non-squares.  Further, if $d$ is odd (so that $\la(2)=-1$), an equivalence class containing $\be \in \F\setminus\Ffive$  which contains 4  non-zero squares or 4 non-squares implies that the equivalence class which contains $2\be$ must contain 4 non-squares or 4 non-zero squares, respectively, i.e., a quadruple of consecutive non-zero squares is associated with one comprising non-squares.  The result follows.
\end {proof}

A character sum involving a product of four linear factors is a new consideration..

\begin{lemma}\label{4lin}
 Define $\displaystyle{V= \sum_{\al \in \F}\la(\al(\al-1)(\al+1)\al+2)})$.  Then
	\[ V=\begin{cases}-2s-1, \mbox{ if   d  is even},\\
\;\; \;2 s-1,\,  \mbox{ if   d  \textrm  is odd}, 
\end{cases}\]
where $s$ is as in Case $3$ of Theorem $ \ref{main}$. 

\end{lemma}
\begin{proof}Replacing $\al$ by $\al+2$, we have
\begin{eqnarray}\label{5V}
V&=&\sum_\al \la((\al+1)( \al+2)( \al+3)(\al+4))=\sum_\al\la((\al+1)(\al-1)(\al+2)(\al-2)) \nonumber\\
&=&\sum_\al \la((\al^2- 1)(\al^2-4))=\sum_\al \la((\al^2- 1)(\al^2+1)).
\end{eqnarray}
Now, in the sum of (\ref{5V}), set $\be=\al^2$ and multiply by the weighting factor $1+\la(\be)$ which (correctly) assigns 2 to each non-zero square  $\be$ (corresponding to $\pm\al$) and 1 for $\be=0$.  Thus
\[V= \sum_{\be} \la(\be^2-1)+ \sum_\be\la((\be(\be^2-1))= -1+J(-1), \]
by Lemma \ref{BEW},  and the result follows from Lemma \ref{KR}.
\end{proof}

The remainder of the proof parallels that of  Theorem \ref{main} in Section \ref{secproof}.   Corresponding to (\ref{Meq}) and (\ref{Neq}), we have
\begin{equation}\label{Meq1}
m_q  = \frac{1}{16}\sum_{\alpha \neq 0, \pm 1,-2}(1+\la(\alpha))(1+\la(\alpha-1))(1+\la(\alpha+1) )(1+\la(\al+2)), 
\end{equation}
and 
\begin{equation}\label{Neq1}
n_q  = \frac{1}{16}\sum_{\alpha \neq 0, \pm 1,-2}(1-\la(\alpha))(1-\la(\alpha-1))(1-\la(\alpha+1) )(1-\la(\al+2)).
\end{equation} 

Thus
 \begin{equation}\label{Meq2}
m_q= \frac{1}{16}\left( q-4+ \sum_{i=1}^4S_i+ \sum_{i=1}^6 T_i+ \sum_{i=1}^4U_i+V\right)
\end{equation}
and
 \begin{equation}\label{Neq2}
n_q= \frac{1}{16}\left( q-4- \sum_{i=1}^4S_i+ \sum_{i=1}^6 T_i- \sum_{i=1}^4U_i+V\right).
\end{equation}
Here $S_i$, $T_i$, $U_i$ are sums of the form $\displaystyle{\sum_{\al \ne 0, \pm1, -2}f(\al)}$, where the polynomial $f$ runs through all factors  of the polynomial $f(x)$ of degrees  1,2,3, respectively. Further $V$ is as defined in Lemma \ref{4lin} in which case  the restrictions on $\al$ are immaterial since $\la(0)=0$.   We evaluate these in turn.

\medskip
First the $S_i$.  We have
\begin{eqnarray*}
\sum_{\al \ne 0,-1,-2}\la(\al-1)&=&\sum_\al\la(\al-1) -1- \la(-2) -\la(2)=-1  -2 \la(2); \\
\sum_{\al \ne \pm1,-2}\la(\al)\qquad& =&\sum_\al\la(\al)-1 -1 - \la(-2)= -2- \la(2); \\
\sum_{\al \ne0, 1, -2}\la(\al+1) &=&\sum_\al\la(\al +1)-1-\la(2) -1=-2- \la(2); \\
\sum_{\al \ne0, \pm1}\la(\al+2)& =&\sum_\al\la(\al +2) -\la(2) -\la(-2)-1=-1-2 \la(2). \\
\end{eqnarray*}
Hence $S=\ds{\sum_{i=1}^4S_i} =-6 - 6\la(2)$.  Thus $S=-12$ if $d$ is even but $S=0$ if $d$ is odd.

\medskip
Similarly, using Lemma \ref{BEW},  we have
 $T=\ds{\sum_{i=1}^6 T_i} =-10- 8\la(2)$.  In particular, $T=-18$, if $d$ is even, whereas $T=-2$ if $d$ is odd.
  
 Finally,  for the $U_i$, Lemma \ref{KR} comes into play again.  We have
  \begin{eqnarray*}
  \sum_{\al\ne -2} \la(\al(\al-1)(\al+1))&=&  \sum_{\al} \la(\al(\al-1)(\al+1))- 1\\
  &=&J(-1) -1\\
  &=&\begin{cases} -2s-1, \mbox{ if $ d$ is even},\\
 \;\; 2s-1,\; \mbox{ if $ d$ is odd}.
  \end{cases}
  \end{eqnarray*}  Next, replacing $\al$ by $\al-2$ in the sum,
\begin{eqnarray*}
  \sum_{\al \ne -1} \la(\al(\al-1)(\al+2))&=& \sum_{\al } \la(\al(\al-1)(\al+2))-\la(-2)\\
  &=& \sum_\al\la((\al-2)(\al+2)\al)-\la(2)\\
  &=&\sum_\al\la(\al(\al^2-4)) - \la(2)\\
  &=&J(1) - \la(2)= -2s- \la(2).
  \end{eqnarray*}
  Next,  again replacing $\al$ by $\al-2$ in the sum,
  \begin{eqnarray*}
  \sum_{\al\ne -1}\la((\al-1)\al(\al+2))&=& \sum_\al \la((\al+2) (\al-2)\al) -\la(2)\\
  &=&J(1)- \la(2)=-2s- \la(2).
  \end{eqnarray*}
  Finally, replacing $\al$ by $\al-1$ in the sum,
  \begin{eqnarray*}
  \sum_{\al \ne 1} \la(\al(\al+1)(\al+2))&=& \sum_\al \la((\al-1)\al(\al+1))- 1\\
  &=&J(-1)-1= \begin{cases} 
  -2s-1,  \mbox{ if $d$ is even},\\
\;\;  2s-1, \mbox{  if  $d$ is odd}.
  \end{cases}
  \end{eqnarray*}
  Set $U= \ds{\sum_{i=1}^4 U_i}$.  From the above, if $d$ is even, then $U=-8s-4$.  On the other hand, if $d$ is odd, then $U=0$.

\medskip
We put all this together using (\ref{Meq1}) and (\ref{Neq1}).  First, suppose $d$ is even.  Then
\[16m_q =q-4-12 -18-8s-4-2s-1= q=10s-30.\]
Further,
\[16n_q=q-4 -12-18-8s-4-2s-1=q+6s-7.\]
This yields (\ref{5even}).

\medskip
Now suppose $d$ is odd.  By Lemma \ref{M-N} we already know that $m_q= n_q$ (Lemma \ref{M-N}).   This is consistent with the fact that here $S=U=0$. In fact, from the above,
\[16n_q=q-4 -2+2s-1=q+2s-7, \]
in agreement with (\ref{5odd}).

This completes the proof of Theorem \ref{5thm}.

\medskip
 Given the apparently arbitrary form of the expressions  in Theorem \ref{5thm}, it is reassuring to present  some illustrative examples for small values of $d$.  A preliminary observation is that, if $d$ is even, then the basic quadruple comprises 4 consecutive non-zero squares so that $m_q \ge 1$.
 
 If $d=1$, then the basic quadruple is the only quadruple of consecutive elements and clearly $m_5=n_5=0$.  Since $5=1^1+2^2$ this agrees with  (\ref{5odd}) in this case.
 
 \medskip
 If $d=2$, then $q=25=(-3)^2+4^2$, so that $s=-3$.  Thus, from (\ref{5even}), $m_5=1, n_5=0$.  In particular, the basic quadruple is the only one consisting of squares.
 
 \medskip
 At this point, we recall that $ M_{25}=N_{25}=2$ and  pause  to identify explicitly the  triples of squares and non-squares in  $\Ftwofive$.  First, the basic quadruple $\Q_2$ of consecutive squares yields two triples $\T_2$ and $\T_3$ of consecutive non-zero squares.  Next, apart from the equivalence  class $\Ffive$  in $\Ftwofive$ there are four further equivalence classes which we now describe.   Let $\Ftwofive=\Ffive({\be)}$, where $\be^2=2$.   Then $\be$ has order 8 and is therefore a non-square in $\Ftwofive$. Further,  all  of $i\be, i=\pm1,\pm2$ are inequivalent and  the four classes are $i\be+\Ffive=i(\be+\Ffive), i= \pm1, \pm2$.  Moreover, $(\be-2)^2=\be+1$ and so the latter is a square.   Since $\be^2-1=1$, a square, then $\be-1$ is also a square.  Also $(\be-2)^3=-\be$ which has order 8 and so $\be-2$ is a non-square.  So is $\be+2$, since $ \be^2-4=-2$, a square.   Thus each of the four equivalence classes comprise three non-squares and two squares.  There are therefore no more quadruples of consecutive square or non-squares and no further triples of consecutive squares.   Observe also that $-(\be+\Ffive)$ is just the set of negatives of $\be+\Ffive$ and similarly for $\pm2(\be+\Ffive)$.  In $\be+\Ffive$ the non-square $\be$ lies between the squares $\be \pm1$ and so  each of $\pm(\be +\Ffive)$ does not contain a triple of consecutive non-squares, whereas, in    $\pm2(\be+\Ffive)$ the squares $2\be-2= 2 \be+3$ and $2\be+2$ are consecutive, thus  confirming that $M_{25}=2$. 
 
 \medskip
 We resume the consequences of Theorem \ref{5thm} for a few more values of $d$.
 If $d=3$, then $q=125=(-11)^2+2^2$.  Thus, from (\ref{5odd}), $m_{125}=n_{125}=16$.
 
 If $d=4$, then $q=625=(-7)^2+24^2$.  Thus $m_{625}=41,\ n_{625}= 36$.
 
 If $d=5$, then $q-3125=41^2+38^2$.  Thus $m_{3125}=n_{3125}=200$.
 
 If  $d=6$, then $q=15625=117^2+44^2$.   Thus $ m_{15625}=901,\  n_{15625}=1020$.
 
 \medskip
 
 From (\ref{5even}), if $d$ is even then, since  $|s|< \sqrt{q}$,   $n_q$ is positive iff $  (\sqrt{q} -3)^2> 16$, so certainly if 
 $q>49$.  In particular, there exists a quadruple of non-squares whenever $d  \geq 4$.         On the other hand  $m_q>1$  iff  $(\sqrt{q}-5)^2>80$ which occurs if $q>195$.   Hence,  there exists a quadruple of consecutive squares  (other than the basic quadruple) iff $d \geq  4 $.
 
 Further, if $d$ is odd, there exists a quadruple of non-zero squares or of non-squares whenever $(\sqrt{q}-1)^2 >8$ and so whenever $q>15$ which occcurs whever $d \geq 3$.
 
 Finally, when $q=5^d$, the existence of a quadruple of consecutive primitive elements of $\F$ is established in \cite{JT22} for all $d\ge 3$.     This is a deep result.

\section*{Acknowledgements}
We are grateful to the referee for a number of suggestions and corrections.


\begin{thebibliography}{99}
 %\vskip4.0pc
 \bibitem{ACDT23}
A.N.  Abyzov, S. D. Cohen, P. V.  Danchev and D. T. Tapkin, Rings and finite fields whose elements are sums or differences of tripotents and potents,  {\it Turkish J. Math}, (2024), to appear.

\bibitem{smz-2021}
A. N.  Abyzov and D. T. Tapkin,   Rings over which every matrices are sums of idempotent and $q$-potent matrice., {\it  Siberian Math. J.} \textbf{62} (2021) 1--13.

\bibitem{laa-2021}
A. N. Abyzov and D. T. Tapkin, When is every matrix over a ring the sum of two tripotents?, {\it  Lin. Algebra  Appl.}  \textbf{630} (2021)  316--325.


\bibitem{BEW98}
 B. C. Berndt, R. J.   Evans and K. S.   Williams, {\it  Gauss and Jacobi Sums} Canadian Mathematical Society Series of Monographs and Advanced Texts (John Wiley \& Sons, Inc., New York, 1998).


\bibitem{BH84}
D. A. Buell and R.H.  Hudson,    On runs of consecutive quadratic residues and quadratic nonresidues, {\it  BIT Numerical Mathematics} \textbf{24} (2)  (1984)  243--247.

\bibitem{C85}
S. D. Cohen,   Consecutive primitive roots in a finite field, {\it  Proc. Amer. Math.  Soc.}  \textbf{93} (1985)  189--197.
	
\bibitem{CST15}
S. D. Cohen,  T.O.  e Silva, and T.S.  Trudgian,   On consecutive primitive elements in a finite field. {\it  Bull. Lond. Math. Soc.}  \textbf{47},  418--426 (2015).

\bibitem{JT22}
T.  Jarso and  T,  Trudgian, T.: Four consecutive primitive elements in a finite field, {\it  Math. Comp.} {\bf 91} (2022) 1521--1532.
                
\bibitem{KR87}
S. A.  Katre and A. R.  Rajwade,   Resolution of the sign ambiguity in the determination of the cyclotomomic numbers of order $4$ and the corresponding Jacobsthal sum, {\it  Math. Scand.} \textbf{60} (1987)  52--62.

\bibitem{KTVZ}
V.  Kirichenko, M. Tsfasman, S.  Vladut and I. Zakharevich,: Quadratic residue patterns and point count on K3 surfaces,  {\it arXiv math.}  2303.03270 (2023)  9 pages. 
	
\bibitem{M85}
M. G.  Monzingo,  On the distibution of consecutive triples of quaratic residues and quadratic nonresidues and related topics, {\it  Fibonacci Quart.} \textbf{23} (1985) 133--138.

\bibitem{M86}
M. G. Monzingo,  An elementary evaluation of the Jacobsthal sum,  {\it  J. Number Th.} {\bf 22} (1986)  21--25.

\bibitem{S02}
Z.-H. Sun,  Consecutive numbers with the same Legendre symbol, {\it Proc. Amer. Math. Soc, } \textbf{190}, (2002) 2503-2507.

\end{thebibliography}
\end{document}